\documentclass[11pt,reqno]{amsart}
\usepackage{hyperref}
\usepackage{amsmath,amssymb}
\usepackage{xcolor}
\usepackage{amsthm}
\usepackage{centernot}
\theoremstyle{plain}
\usepackage[pagewise]{lineno}
\newtheorem{thm}{Theorem}[section]
\newtheorem{lemma}[thm]{Lemma}
\newtheorem{cor}[thm]{Corollary}
\newtheorem{prop}[thm]{Proposition}

\theoremstyle{definition}
\newtheorem{defn}[thm]{Definition}
\newtheorem{example}[thm]{Example}

\numberwithin{equation}{section}

\newcommand{\interior}[1]{{\kern0pt#1}^{\mathrm{o}}}

\makeatletter
\@namedef{subjclassname@2020}{%
	\textup{2020} Mathematics Subject Classification}
\makeatother

\begin{document}
	
	\title[Menger Convexity and Fixed Point Results]
	{Menger Convexity and Fixed Point Results}
	
		\author[A. Gupta]{Ajit Kumar Gupta$^{1 *}$}
	\address{$^1$Department of Applied Science and Humanities\\ Sankalchand Patel College of Engineering\\ Visnagar 384315\\ India}
	\email{emailstoajit@gmail.com, ajitguptahm\_spce@spu.ac.in}
	
	\author[S. Mukherjee]{Saikat Mukherjee{$^{2}$}}
	\address{$^2$Department of Mathematics\\ National Institute of Technology Meghalaya\\ Sohra 793108\\ India}
	\email{saikat.mukherjee@nitm.ac.in}
	$\thanks{*Corresponding author}$
	
	
	\begin{abstract}
		Shimizu and Takahashi proved that every decreasing sequence of nonempty, bounded, closed, convex subsets of a complete, uniformly Takahashi convex metric space has nonempty intersection. It is well known that the Menger convexity is a generalization of the Takahashi convexity. In this article, we acquire a nonempty intersection property, in terms of the Hausdorff metric, for Menger convex metric spaces, that also provides a class of reflexive Menger convex spaces. We introduce a generalization of $(\alpha, \beta)-$generalized hybrid mappings, and using the obtained nonempty intersection property we derive the fixed point results for this generalized mapping defined on Menger convex spaces.
	\end{abstract}
	
	\maketitle

\section{Introduction}
The concept of convex structure on a space is not limited only to a vector space but the researchers have extended it to the metric spaces, even to the topological spaces, in recent years. In 1970, Takahashi (see \cite{wt70}) introduced a convex structure on a metric space that is a generalization of the convex structure in the normed spaces. Earlier to it, in 1928, Menger (see \cite{km28}) proposed a concept of convexity in metric spaces, which is known as Menger convexity and is a generalization of the Takahashi's convex structure. Moreover, the metrically convex structure (see \cite{am04}) is weaker than the Menger's convex structure. In 1997, Taskovic (see \cite{mr97, sn03, mr08}) introduced a convex structure on the topological spaces, which is known as general convex structure and is weaker than Takahashi's convexity.

Several remarkable nonempty intersection properties for the metric spaces with some convex structure have been discovered by the researchers.  Klee's Theorem (see \cite{vl56}) says that every infinite dimensional normed space has a decreasing sequence of unbounded but linearly bounded, closed, convex subsets with nonempty intersection. Smulian's theorem (see \cite{nc12}) characterizes the reflexivity of normed spaces.  According to it, a normed space $X$ is reflexive if and only if every decreasing sequence of nonempty, closed, bounded, convex subsets of $X$ has nonempty intersection. Further, in a more general convex structure i.e. in a complete uniformly Takahashi  convex metric spaces $Y$, Shimizu and Takahashi (see \cite{ts96}) have found that, every decreasing sequence of nonempty, bounded, closed, convex subsets of $Y$ has nonempty intersection.

Kocourek et al., in \cite{pk10}, have shown that for a given nonempty, closed, convex subset $C$ of a Hilbert space, an $(\alpha, \beta)$-generalized hybrid mapping $T:C\to C$ has a fixed point if and only if the sequence $\{T^nx\}$ is bounded for some $x\in C$. Later, Phuengrattana et al., in \cite{wp14}, extended this result to Takahashi convex structure.

However, no researchers have studied the nonempty intersection properties and fixed point results for $(\alpha, \beta)$-generalized hybrid mapping in Menger convex structure which is more general than Takahashi convex structure. In this article, we obtain nonempty intersection properties for Menger convex metric spaces, which also imply that the Menger convex metric spaces with a certain property are reflexive. 
 We introduce a mapping weaker than the $(\alpha, \beta)$-generalized hybrid mappings, and using our obtained results, we derive fixed point results for this introduced mapping, defined on a nonempty, closed, convex subset of a Menger convex metric space.
 
  The manuscript is organized as follows. Section \ref{MengerPrelim} contains the preliminaries required for the discussion in subsequent sections. In Section \ref{TakCon}, we discuss the Takahashi convexity in the context of Menger convexity. In Section \ref{NIPMB'}, we derive some nonempty intersection results for Menger convex metric spaces. In Section \ref{FPR}, we introduce a mapping which is a generalization of $(\alpha, \beta)-$generalized hybrid mappings; and using the results obtained in the previous sections, we acquire the fixed point results for this introduced mapping.

\section{Preliminaries}\label{MengerPrelim}
	
\begin{defn}
	Given a metric space $(X,d)$, a mapping $W:X^2\times [0,1]\to X$ is said to be a \textit{Takahashi convex structure}  on $X$ if, $\forall~ x,y,u\in X$ and $t \in [0,1]$, we have
	$d(u,W(x,y,t))\leq t d(u,x)+(1-t)d(u,y)$. The space $X$ together with a convex structure $W$ is called a \textit{Takahashi convex metric space}.
\end{defn}
Normed spaces and CAT(0) spaces (see \cite{sm09}, \cite{pc06}, \cite{sd08}) are Takahashi convex metric spaces.

\begin{defn}
	A Takahashi convex metric space $X$ is said to be \textit{uniformly convex} (\hspace{1sp}\cite{ts96}) if, for each $\epsilon>0$ there is a $\delta(\epsilon)>0$ such that for $c>0$, $x,y,z \in X$ with $d(z,x)\leq c, d(z,y)\leq c$ and $d(x,y)\geq c\epsilon$, we have $d(z,W(x,y,1/2))\leq c(1-\delta)$.
\end{defn}

\begin{defn}
	A metric space $(X,d)$ is said to be \textit{Menger convex} (\hspace{1sp}\cite{ib05, km28}) if for each $x,y\in X$ and $0< t < 1$, the set $B[x,td(x,y)]\cap B[y,(1-t)d(x,y)]\neq \emptyset $, where $B[x,r]$ denotes the set $\{y\in X:d(x,y)\leq r\}$, $r\geq 0$.
\end{defn}
 
 If for some $x,y$ in a Menger convex metric space and for some $t\in (0,1)$, the set $B[x,td(x,y)]\cap B[y,(1-t)d(x,y)]$ is singleton, then we denote this set by $m(x,y,t)$.

 Using Lemma 4.7 in \cite{ag23}, one can replace the set $B[x,td(x,y)]\cap B[y,(1-t)d(x,y)]$ with the set $S[x,r]\cap S[y,d(x,y)-r]$ in the above definition, where $S[x,r]=\{y\in X: d(x,y)=r\}$, $0\leq r \leq d(x,y)$.
 
 \begin{defn}
 	A Menger convex metric space $X$ is said to have the \textit{property $(A)$}, if for each $x,y\in X$ and $0< t < 1$, the set $m(x,y,t)$ exists.
 \end{defn}
 

\begin{defn}
		A subset $A$ in Takahashi convex metric space is said to be \textit{convex} if $W(x,y,t)\in A$ for all $x,y\in A$, $t\in [0,1]$. 
\end{defn}
\begin{defn}
	A subset $A$ in Menger convex metric space is said to be \textit{convex} if, $\forall x,y\in A$ and $0<r<d(x,y)$, the set $B[x,r]\cap B[y,d(x,y)-r]\subset A$.
\end{defn}

Closed balls in a Takahashi convex metric space are convex. But closed balls in a Menger convex metric space need not be convex; they are convex if the space has the property $(A)$.
There are some Takahashi convex metric spaces which don't have the property $(A)$, for example, the normed spaces $(l_1, \|\cdot\|_1)$ and $(l_{\infty},\|\cdot\|_{\infty})$ do not have the property $(A)$ (see the examples \ref{l1dnhpa} and \ref{l_inftydnhpa}); and in these normed spaces, when treated as Menger convex metric spaces, the closed ball $B[0,1]$ is not convex.

\begin{defn}
	A Menger convex metric space $X$ is said to satisfy the \textit{property $(B)$} (see \cite{ib05}) if, $  \forall~ x, y, z \in X,\ t \in (0, 1)$, we have $d(m(x,y,t),$ $ m(z,y,t)) \leq td(x,z)$.
\end{defn}

\begin{defn}\label{dmu1} 
	A Menger convex metric space $(X,d)$ with the property $(A)$ is said to be \textit{uniformly convex} (\hspace{1sp}\cite{ib05}) if for each $\epsilon>0$ there exists a $\delta(\epsilon) >0$ such that $\forall~ r>0$ and $\forall~ x, y, z \in X$ with $d(z,x) \leq r,\ d(z,y) \leq r$ and $d(x,y) \geq r\epsilon$, we have
	$d(z,m(x,y,\frac{1}{2}))\leq (1-\delta)r.$
\end{defn}

\begin{defn}
	A metric space $X$ with a convex structure is said to be \textit{reflexive} (\hspace{1sp}\cite{re11}) if each decreasing sequence of nonempty closed bounded convex subsets of $X$ has nonempty intersection.
\end{defn}

A reflexive Banach space is an example of reflexive metric space. Other examples for reflexive spaces include complete CAT(0) spaces, and complete uniformly Takahashi convex metric spaces (see \cite{ts96}). In Section \ref{NIPMB'}, we will have another example of reflexive spaces, with Menger convex structure.

	\begin{defn} Given a pair of nonempty subsets $A,B$ of a metric space $(X,d)$, the {\it Hausdorff distance} (\cite{ag23}) between them is defined as $$H(A,B)=\max\left\{\sup\limits_{x\in A}d(x,B), ~\sup\limits_{x\in B}d(x,A)\right\},$$ where $d(x,A)=\inf\limits_{y\in A}d(x,y)$.
	\end{defn}
The distance function $H$ is a metric, provided the sets $A,B$ are closed and bounded.
\section{Takahashi Convexity}\label{TakCon}

In this section we explore Takahashi convexity in the perspective of Menger convexity. 
\begin{thm}\label{uwcmspa} 
	Uniformly Takahashi convex metric spaces have the property $(A)$.
\end{thm}
\begin{proof} Consider a pair of distinct points $x,y$ in a uniformly Takahashi convex metric space $X$.  If possible, suppose, for some $t\in (0,1)$, $z_1,z_2\in B[x,td(x,y)]\cap B[y,(1-t)d(x,y)]$ are two distinct points. By Lemma 4.7 in \cite{ag23}, we have $z_1,z_2\in S[x,td(x,y)]\cap S[y,(1-t)d(x,y)]$. Now, we can choose an $\epsilon$ such that
	$$0<\epsilon<\min \bigg\{\frac{d(z_1,z_2)}{td(x,y)},\frac{d(z_1,z_2)}{(1-t)d(x,y)}\bigg\}.$$
	This implies $$d(z_1,z_2)>td(x,y)\epsilon>0,$$
	$$d(z_1,z_2)>(1-t)d(x,y)\epsilon>0.$$
	We have,
	$$d(x,z_1)=d(x,z_2)=td(x,y),$$
	$$d(y,z_1)=d(y,z_2)=(1-t)d(x,y).$$
	Now, by the uniform convexity, for the chosen $\epsilon>0$ $\exists~ \delta>0$ such that
	$$d(x,W(z_1,z_2,1/2))\leq td(x,y)(1-\delta),$$
	$$d(y,W(z_1,z_2,1/2))\leq (1-t)d(x,y)(1-\delta).$$
	Adding these two inequalities we get:
	$$d(x,y)\leq d(x,W(z_1,z_2,1/2))+d(y,W(z_1,z_2,1/2))\leq (1-\delta)d(x,y),$$
	which implies $d(x,y)<d(x,y)$ or $d(x,y)\leq 0$, and either of these is a contradiction. Hence, the theorem is proved.
\end{proof}

We shall use Theorem \ref{uwcmspa} to derive a fixed point result in Section \ref{FPR}. 

\begin{cor}
	Uniformly Takahashi convex metric spaces are uniformly Menger convex.

\end{cor}
\begin{cor}\label{ucbspa}
	Uniformly convex normed spaces have the property $(A)$.
\end{cor}

Corollary \ref{ucbspa} implies that, given $p\in (1,\infty), ~ a,b \in \mathbb R$, the normed spaces $l_p$ and $L_p[a,b]$  satisfy the property $(A)$ as these spaces are uniformly convex. The normed spaces $L_1[a,b]$, $L_\infty[a,b]$, $l_1$, and $l_\infty$ do not have the property $(A)$, which is evident from the following examples. We shall take the help of Lemma 4.7 in \cite{ag23}, to discuss the following examples.

\begin{example}\label{L1dnhpa}
	Let $f(x)=0$ and $g(x)=2$ be in the normed space $X=(L_1[0,1],\|\cdot\|_1)$, and $t=\frac{1}{2}.$ Let $a\in (0,1)$ and a mapping $h_a$ on $[0,1]$ be defined as: $h_a(x)= \frac{-2}{a}x+2, x\in [0,a]$, and $h_a(x)=\frac{2}{1-a}(x-1)+2, x\in [a,1]$. Then, the set $\{h_a\in X \ |\ a \in (0,1)\}\subset S[f,td(f,g)]\cap S[g,(1-t)d(f,g)]$. Thus, the normed space $X$ does not have the property $(A)$.
\end{example}
\begin{example}\label{L-inftydnhpa}
	Let $f(x)=0$ and $g(x)=2x^2$ be in the normed space $X=(L_{\infty}[0,1], \|\cdot\|_\infty)$, and $t=\frac{1}{2}$. Let $P=\{p:[0,\frac{1}{\sqrt 2}] \to [0,1]\ |$ p is a continuous mapping with $p(0)=0, p(\frac{1}{\sqrt 2})=\frac{1}{2}\}$  and a mapping $h_p$ on $[0,1]$ be defined as:
	$h_p(x)=p(x), x\in [0,\frac{1}{\sqrt 2}]$, and $h_p(x)=x^2, x\in [\frac{1}{\sqrt 2},1]$. Then, the set $\{h_p\in X\ |\ p \text{ is in } P\}\subset S[f,td(f,g)]\cap S[g,(1-t)d(f,g)]$. Hence, the space $X$ does not have the property $(A)$.
\end{example}

\begin{example}\label{l1dnhpa}
	Let  $ n \in \mathbb N, n\geq 2$, be fixed and $x=(0,0,0,...)$, $ y=(\underbrace{2/n, 2/n,...,2/n,}_\text{n times}$ $0,0,0,...)\in (l_1, \|\cdot\|_1)$.
	For some given $1\leq i,j\leq n, i\neq j$, let $a_{ij}=(x_1,x_2,...,x_i,...,x_n,  $ $0,0,0,...)$ such that $  x_i=0, x_j=2/n,\ \text{and}\ x_p=1/n \ \forall\ p\neq i,j.$
	Then, the set $\{a_{ij}\in l_1:1\leq i,j\leq n, i \neq j\}\subset S[x, d(x,y)/2]\cap S[y, d(x,y)/2].$
	Thus, the space $(l_1, \|\cdot\|_1)$ does not have the property $(A)$.
	
\end{example}

\begin{example} \label{l_inftydnhpa}
	Let $x=(0,0,0,...)\ \text{and}\ y=(2,0,0,0,...)$ be in the normed space $(l_{\infty},\|\cdot\|_{\infty}).$ Then, for $t=\frac{1}{4}$,
	$S[x,t_1d(x,y)]\cap S[y,(1-t_1)d(x,y)]= \{(x_i)\in X: x_1=\frac{1}{2},\ -\frac{1}{2}\leq x_i\leq \frac{1}{2}\ \forall\  i\geq 2\}.$
	This implies, the space $(l_\infty, \|\cdot\|_\infty)$ does not have the property $(A)$.
	
\end{example}

\begin{prop}\label{mtvs}
	Consider a real vector space $X$ endowed with a metric $d$, and with a Takahashi convex structure $W(x,y,\lambda)=\lambda x +(1-\lambda)y$. If the metric $d$ is translation invariant, then the space $(X,d)$ is a normed space.
\end{prop}
\begin{proof} By the hypothesis we have,
	$d(0,\lambda x)\leq \lambda d(0,x), \forall~ x \in X, \lambda \in [0,1].$ We shall prove $d(0,\alpha x) = |\alpha|d(0,x), \forall x\in X, \alpha \in \mathbb R$. Let $p \in \mathbb R$. If $p\geq 1$, then $p = n+r,\ \text{for some}\ n\in \mathbb N , r\in[0,1)$. For an $x\in X$, we have
	\begin{equation*}
	\begin{aligned}[b]
	d(0,(p -1)x)
	& \leq d(0,x)+d(x,(p -1)x)  \\
	& =  d(0,x)+d(0,(p -2)x) \\
	& \leq  2d(0,x)+d(0,(p-3)x) \\
	&\ ... \\	
	&\leq (n-1)d(0,x)+d(0,rx) \\
	&\leq (n-1)d(0,x)+rd(0,x) \\
	&= (p -1)d(0,x).
	\end{aligned}
	\end{equation*}
	This implies,
	\begin{equation}
	d(0,q x)\leq qd(0,x),\ \forall~ q \geq 0. \label{eqn.1}
	\end{equation}
	Since the metric $d$ is translation invariant, using (\ref{eqn.1}), for $p \leq 0$, we have,
	$d(0, p x)= d(0, -p x) \leq |p|d(0,x).$
	Thus, finally we have
	$d(0,p x)\leq |p|d(0,x), \forall x\in X, p \in \mathbb R.$ Now, for $\alpha \in \mathbb R \setminus \{0\}$, we have
	$$	d(0,x)=d(0,\frac{\alpha}{\alpha}x)\leq \frac{1}{|\alpha|}d(0,\alpha x).$$
	And therefore,
	$$|\alpha|d(0,x)\leq d(0,\alpha x)\leq |\alpha|d(0,x), \forall \alpha \in \mathbb R.$$
\end{proof}

\section{Nonempty Intersection Results}\label{NIPMB'}

Throughout this section, by $X$ we mean a Menger convex metric space $(X,d)$.

	We extend the domain of the property $(B)$ defined in Section \ref{MengerPrelim}, as follows.
	\begin{defn}\label{dpb'}
		$X$ is said to satisfy the \textit{property $(B')$} if $\forall~ x, y, z\in X, t\in (0,1)$, we have
		$$H(c(x,y,t),c(z,y,t))\leq t~d(x,z),$$
		where $c:X^2 \times (0,1)\to C(X)$ is defined as
		$c(x,y,t)=B[x,td(x,y)]\cap B[y,(1-t)d(x,y)].$
	\end{defn}
	
	\begin{thm}\label{NIPB'}
		Consider a complete space $X$ with the property $(B')$. Then, for each decreasing sequence $\{K_n\}$ of nonempty, bounded, closed, convex subsets of $X$, we have
		$\bigcap\limits_{n=1}^\infty K_n\neq\emptyset$.
	\end{thm}
	\begin{proof}
		Consider a point $z_n\in K_n$, and a pair of points $x_{n+1}, y_{n+1}\in K_{n+1}$. Then, by the property $(B')$, we have $H\big(c(x_{n+1},y_{n+1},\frac{1}{3^{n+1}}),c(z_n,y_{n+1},$ $\frac{1}{3^{n+1}})\big)$ $\leq \frac{1}{3^{n+1}}d(x_{n+1},z_n)$ $\leq \frac{1}{3^{n+1}}\delta (K_1)$. Let us pick a point $z'_{n}\in c(z_n,y_{n+1},$ $\frac{1}{3^{n+1}})$. Then, there is a point, say $z_{n+1}$, in $c(x_{n+1},y_{n+1},\frac{1}{3^{n+1}})\subset K_{n+1}$ such that $d(z_n', z_{n+1})<\frac{1}{2^{n+1}}\delta(K_1)$. Now, $d(z_n,z_{n+1})\leq d(z_n,z_n')+d(z_n',z_{n+1})< (\frac{1}{3^{n+1}}+\frac{1}{2^{n+1}})\delta(K_1)$, which implies the sequence $\{z_n\}$ is Cauchy. And hence, its limit is in $\bigcap\limits_{n=1}^\infty K_n$.
	\end{proof}

	The above theorem also implies that the spaces $X$ with property $(B')$ are reflexive.
	
	It can be observed that, if the space $X$ has the property $(A)$, then the properties $(B)$ and $(B')$ are equivalent.
	
	\begin{cor}
		Consider a complete space $X$ with the properties $(A)$ and $(B)$. Then, for each decreasing sequence $\{K_n\}$ of nonempty bounded closed convex subsets of $X$, we have
		$\bigcap\limits_{n=1}^\infty K_n\neq\emptyset$.
	\end{cor}

	
	%

\section{Fixed Point Results}\label{FPR}
 
In this section, we will discuss the fixed point results for a mapping of Menger convex spaces which is weaker than the $(\alpha, \beta)-$ generalized hybrid mappings. It may be noted that the result for Takahashi convex metric space similar to Lemma \ref{u0b'} can be found in the work of Phuengrattana et al. (see \cite{wp14}).

\begin{defn} 
	Let $U$ be a convex subset of a Menger convex metric space $(X,d)$. Then, a mapping $f:U \to \mathbb R$ is said to be \textit{Menger convex function} (\hspace{1sp}\cite{vp79}) if $\forall~ x, y \in U, z \in B[x,td(x,y)]\cap B[y,(1-t)d(x,y)], t\in (0,1)$, we have
	$f(z)\leq (1-t) f(x)+tf(y).$
\end{defn}
The distance function $d$ in a Menger convex metric space $(X,d)$ is said to be Menger convex if $d(u,\cdot):X\to \mathbb R$ is a Menger convex function on $X$ for each $u\in X$. In a Takahashi convex metric space $(X,d)$ satisfying the property $(A)$, the metric $d$ is a Menger convex function; while in the Takahashi convex metric spaces $(l_1,\|\cdot\|_1)$ and $(l_{\infty},\|\cdot\|_{\infty})$, the mapping $d(u,\cdot)$ is not a Menger convex function.

\begin{defn}
	A real-valued function $f$ on a subset $U$ of a metric space $X$ is said to be $d$\textit{-coercive} if, whenever for a sequence $\{x_n\}_{n=1}^{\infty}\subset U$ and $x\in U$, $d(x_n,x)\to \infty$ as $n \to \infty$, then $f(x_n)\to \infty$ as $n \to \infty$.
\end{defn}

\begin{lemma}\label{u0b'}
	Let $U$ be a nonempty, closed, convex subset of a complete, Menger convex metric space $(X,d)$ and $f:U \to [0,\infty)$ be a continuous, d-coercive, Menger convex function. If the space $X$ satisfies the property $(B')$, then there exists a $u_0 \in U$ such that $f(u_0) = \inf\limits_{x\in U}f(x)$.
\end{lemma}
\begin{proof}
	Let $\alpha=\inf\limits_{x\in U}f(x)$ and $K_n:=\{x\in U:f(x)\leq\alpha +\frac{1}{n}\},\ n \in \mathbb N$. The set $K_n$, $n\in \mathbb N$, is nonempty. For, if $K_n=\emptyset$ for some $n\in \mathbb N$, then $f(x)>\alpha +\frac{1}{n} \ \forall~ x \in U$, which implies $\alpha \neq \inf\limits_{x\in U}f(x)$. It is easy to see $K_n's$ are closed and convex. 
	The set $K_n$, $n\in \mathbb N$, is bounded as well. For, if $K_n$ is not bounded for some $n\in \mathbb N$, then there exists a sequence $\{x_m\}_{m=1}^\infty$ in $K_n$ such that $d(x_m,x)\to \infty$ for some $x\in K_n$. This implies $f(x_m)\to \infty$, which is a contradiction.
	Therefore by Theorem \ref{NIPB'},
	$\bigcap\limits_{n=1}^{\infty}K_n\neq \emptyset.$
	Let $u_0 \in \bigcap\limits_{n=1}^{\infty}K_n.$ Then
	$f(u_0)\leq \alpha +\frac{1}{n}~ \forall~ n$, which implies $f(u_0)\leq \alpha$, and hence $f(u_0)= \alpha.$	
\end{proof}

\begin{defn}
	A mapping $T: (X,d)\to (X,d)$ is said to be \textit{metrically nonspreading} (see \cite{fk18}) if $2d(Tx,Ty)^2\leq d(x,Ty)^2+d(Tx,y)^2,\forall~ x, y \in X.$
\end{defn}

\begin{defn}
	 Let $U$ be nonempty subset of a Menger convex metric space $(X,d)$. A mapping $T: U \to U$ is called \textit{$(\alpha, \beta)$-generalized hybrid} (see \cite{pk10}) if there are $\alpha, \beta \in \mathbb R$ such that
	$$\alpha ~d(Tx, Ty)^2 + (1-\alpha)~d(x,Ty)^2 \leq \beta~ d(Tx,y)^2 + (1-\beta)~d(x,y)^2$$
	for all $ x,y \in U$.
\end{defn}
The class of $(\alpha, \beta)-$ generalized hybrid mappings contains the classes of nonexpansive mappings, metrically nonspreading mappings, etc. We introduce a class of mappings that contains the $(\alpha, \beta)-$ generalized hybrid mappings.
\begin{defn}
	Let $U$ be nonempty subset of a Menger convex metric space $(X,d)$. A mapping $T: U \to U$ is called \textit{$(\alpha, \beta, \gamma)$-generalized hybrid} if there are $\alpha, \beta, \gamma \in \mathbb R$ such that
	$$\alpha ~d(Tx, Ty)^2 + (1-\alpha)~d(x,Ty)^2 +\gamma d(Ty,z)^2$$ $$\leq \beta~ d(Tx,y)^2 + (1-\beta+\gamma)~d(x,y)^2+(1+\gamma)d(y,z)^2$$
	 for all $ x,y,z \in U$.
\end{defn}

We notice that, if $y=z$, then $(\alpha, \beta,0)$-generalized hybrid is $(\alpha, \beta)$-generalized hybrid. Therefore, the class of $(\alpha, \beta, \gamma)$-generalized hybrid mappings contains the class of $(\alpha, \beta)$-generalized hybrid mappings. Further, an $(\alpha, \beta, \gamma)$-generalized hybrid map need not be  $(\alpha, \beta)$-generalized hybrid map. For instance, consider $U=\{(0,0,0),(2,0,0),(0,1,0)\}\subset \mathbb R^3$, and a mapping $T:U\to U$ defined as $T(0,0,0)=(0,1,0), T(2,0,0)=(2,0,0), T(0,1,0)=(0,1,0)$. Then $T$ is a $(1,1,1)$-generalized hybrid map on $U$. However, for no $\alpha, \beta\in \mathbb R$, $T$ is $(\alpha, \beta)$-generalized hybrid map; this is evident by choosing $x=(2,0,0), y=(0, 0, 0)$.

\begin{thm}\label{hybridB'}
	Let $U$ be a nonempty, closed, convex subset in a complete, Menger convex metric space $(X,d)$ with the property $(B')$, and the distance function $d$ be Menger convex. Then, an $(\alpha, \beta, \gamma)$-generalized hybrid $T:U\to U$ has a fixed point if and only if the sequence $\{T^nx\}$ is bounded for some $x\in U$.
\end{thm}
\begin{proof} Suppose the sequence $\{T^nx\}$ is bounded for some $x\in U$. We will prove that $T$ has a fixed point. Using Lemma 4.7 in \cite{ag23}, we define a mapping $f:U\to [0,\infty)$ as follows: $$f(y):=\limsup\limits_{n\to \infty} d\big(T^nx,a\big),  \ \text{where } a \in c(T^nx,y,t_1), \text{ for some fixed }t_1\in (0,1).$$
	 
	Let $u,v \in U$ and $ w \in S[u,td(u,v)]\cap S[v,(1-t)d(u,v)]$, where $t\in (0,1)$ is arbitrary. Then,
	\begin{align*}
	f(w)&=t_1\limsup\limits_{n\to \infty} d(T^nx,w)\\
	&\leq t_1\limsup\limits_{n\to \infty} \big[(1-t)d(T^nx,u)+td(T^nx,v)\big]\\
	&\leq t_1(1-t)\limsup\limits_{n\to \infty} d(T^nx,u)+t_1t\limsup\limits_{n\to \infty}d(T^nx,v)\\
	&=(1-t)f(u)+tf(v).
	\end{align*}
	Thus, $f$ is a convex mapping.
	
	By triangle inequality, we have 
	\begin{eqnarray*}
		t_1d(T^nx,u)\leq t_1[d(T^nx,v)+d(v,u)]< t_1d(T^nx,v) + d(u,v).
	\end{eqnarray*}
And therefore, $f(u) < f(v)+d(u,v)$. 	Similarly, we have $f(v) < f(u)+d(u,v).$ These inequalities together impliy that $f$ is uniformly continuous.

	Now, let $\{u_p\}_{p=1}^{\infty}$ be a sequence in $U$ such that $d(u_p,u) \to \infty$ for some $u \in U$.
	Since $d(u,u_p)\leq d(u, T^nx)+d(u_p,T^nx),$
	therefore we have $d(u_p,T^nx)\to \infty$ as $p\to \infty$, for all $n\in \mathbb N$. We claim: The sequence $\{\limsup_{n\to \infty}d(T^nx,u_p)\}_{p\in \mathbb N}$ is not bounded. For, if it is bounded by some $M\in \mathbb R$, then for all $p\in \mathbb N$ and $\epsilon>0$ there is $n_{p,\epsilon}\in \mathbb N$ such that $$d(T^{n_{p,\epsilon}}x,u_p)<\limsup_{n\to \infty}d(T^nx,u_p)+\epsilon\leq M+\epsilon.$$ By triangle inequality, $d(Tx,u_p)\leq d(Tx,T^{n_{p,\epsilon}}x)+d(T^{n_{p,\epsilon}}x,u_p)$, and therefore the sequence $\{d(Tx,u_p)\}_{p\in \mathbb N}$ is bounded, which is a contradiction. And hence, $$\lim_{p \to \infty}[\limsup_{n\to \infty}t_1d(T^nx,u_p)]=\infty$$
	i.e. $f(u_p)\to \infty$ as $p\to \infty$. Therefore, $f$ is a $d$-coercive mapping.
	Hence by Lemma \ref{u0b'}, there exists a $u_0 \in U$ such that $f(u_0) = \inf\limits_{x\in U}f(x)$. We shall prove, this $u_0$ is unique.
	
	If possible, suppose $u_1,u_2$ are two distinct elements in $U$ such that
	$f(u_1) = f(u_2) = \inf\limits_{x\in U}f(x)=\nu$. Then we have, 
	$$ \limsup_{n\to \infty}d(T^nx,p_1)=\limsup_{n\to \infty}d(T^nx,p_2)=\nu,$$ where $p_1\in c(T^nx,u_1,t_1)$ and $p_2\in c(T^nx,u_2,t_1)$.
	This implies that for each $\epsilon \in (0, 1]$ there exists $N \in \mathbb N$ such that
	$$d\big(T^nx, p_1\big) < \nu + \epsilon,\ \  d\big(T^nx, p_2\big) < \nu + \epsilon,\ \forall~ n \geq N.$$ From the definition of Hausdorff metric, for $T^nx$ and $ \epsilon'>0$, there is a $q\in c(p_2,p_1,t_1)$ such that $d(T^nx,q)\leq H(\{T^nx\},c(p_2,p_1,t_1))+\epsilon'.$
	Now, from the property $(B')$, for $n\geq N$
	\begin{align*}
	t_1 d(T^nx,q)&\leq t_1 H(\{T^nx\},c(p_2,p_1,t_1))+\epsilon_1 ,\ \text{ where }\epsilon_1=t_1\epsilon'\\
	&\leq t_1[H(\{T^nx\},c(T^nx,p_1,t_1))+H(c(p_2,p_1,t_1),c(T^nx,p_1,t_1))]+\epsilon_1\\
	&= t_1[H(c(T^nx,T^nx,1-t_1),c(p_1,T^nx,1-t_1))\\
	&\ \ \  +H(c(p_2,p_1,t_1),c(T^nx,p_1,t_1))]+\epsilon_1\\
	&\leq t_1[(1-t_1)d(T^nx,p_1)+t_1d(T^nx,p_2)]+\epsilon_1\\
	&<t_1(\nu+\epsilon)+\epsilon_1.
	\end{align*}
	This implies $f(q) \leq t_1\nu<\nu$, which is a contradiction. Hence, $u_1=u_2$.
	
	Since $T$ is an $(\alpha, \beta, \gamma)$-generalized hybrid, therefore, we have
	$$\alpha d(Tx_0, T^nx)^2 + (1-\alpha)d(x_0, T^nx)^2+\gamma d(T^nx,z_0)^2$$$$
	\leq \beta d(Tx_0,T^{n-1}x)^2 + (1-\beta+\gamma)d(x_0,T^{n-1}x)^2+(1+\gamma)d(T^{n-1}x,z_0)^2,$$
	where $x_0\in c(T^nx,Tu_0,t_1), z_0\in c(T^nx,u_0,t_1)$.
	Since $\alpha, \beta, \gamma \in \mathbb R$, we have
	$$\limsup_{n\to \infty} d(T^nx,x_0)^2 \leq \limsup_{n\to \infty} d(T^nx,z_0)^2.$$
	This implies,	$$ f(Tu_0)\leq f(u_0).$$
	Hence, $Tu_0=u_0.$
	
	The proof for the other direction is immediate.
\end{proof}

\begin{thm}
		Let $U$ be a nonempty, closed, convex subset in a complete, uniformly Takahashi convex metric space $(X,d)$ with the property $(B)$, and the distance function $d$ be Menger convex. Then, an $(\alpha, \beta, \gamma)$-generalized hybrid $T:U\to U$ has a fixed point if and only if there exists an $x \in U$ such that the sequence $\{T^nx\}$ is bounded.
\end{thm}
\begin{proof}
	By Theorem \ref{uwcmspa}, the space $X$ has the property $A$. Hence, the proof follows.
\end{proof}

\subsection{Metric Interval}
\vspace{1cm}
For $x,y$ in a Menger convex metric space, we denote the set $\bigcup\limits_{0\leq t\leq 1}B[x,td(x,y)]\cap B[y,(1-t)d(x,y)]$ by $[x,y]$, and call it a \textit{metric interval}.
\begin{thm}\label{[x,y]}
	For each pair of points $x,y$ in a complete, Menger convex metric space $X$, the metric interval  $[x,y]$ is a complete subset of $X$.
\end{thm}
\begin{proof}
	Let $\{x_n\}\subset [x,y]$ be a Cauchy sequence. Then two cases arise:\\
	Case 1: For all but finitely many $n's$, $\{x_n\}\subset S[x,td(x,y)]\cap S[y,(1-t)d(x,y)] $ for some fixed $t\in (0,1)$: \\
	Since $S[x,td(x,y)]\cap S[y,(1-t)d(x,y)]$ is closed, so $\{x_n\}$ is convergent in $[x,y]$.\\
	Case 2: The negation of the case 1 holds:\\
	So there exists a subsequence $\{x_{p_i}\}\subset \{x_n\}$ which is not in $S[x,td(x,y)]\cap S[y,(1-t)d(x,y)]$ for a fixed $t\in [0,1]$. Then there is a sequence $\{t_i\}\subset [0,1]$ such that
	\begin{equation}
	d(x,x_{p_i})=t_id(x,y)\ \text{and}\ d(y,x_{p_i})=(1-t_i)d(x,y).
	\label{[x,y]1}
	\end{equation}
	
	Then,
	\begin{align*}
	&|t_m-t_n|d(x,y)=|d(x,x_{p_m})-d(x,x_{p_n})|\leq d(x_{p_m},x_{p_n}),
	\end{align*}
	which implies $\{t_i\}\ \text{is a Cauchy sequence in}\ [0,1]$.
	Let $t_i\to t_0$ and $x_{p_i}\to x_0 \in X$. Then by (\ref{[x,y]1}),
	\begin{align*}
	&d(x,x_0)=t_0d(x,y)\ \text{and}\ d(y,y_0)=(1-t_0)d(x,y),
	\end{align*}
	which implies $x_0 \in S[x,t_0d(x,y)]\cap S[y,(1-t_0)d(x,y)]$ and hence the sequence $\{x_n\}$ is convergent in $ [x,y]$.
	
Thus, the set $[x,y]$ is complete.
\end{proof}
The set $[x,y]\subset X$ is bounded, but not totally bounded, e.g.

\begin{example}
	Let $x=e_1,y=2e_2$ be in the metric space $(l_{\infty},\|.\|_{\infty})$ and $t=1/2$. Then $S[x,td(x,y)]\cap S[y,(1-t)d(x,y)]$ contains a sequence
	$e_2, e_2+e_3, e_2+e_4,e_2+e_5,...$
	which does not have a convergent subsequence. Therefore, $[x,y]$ is not compact and hence not totally bounded.
	
\end{example}

A metric interval is not necessarily a Menger convex subset, e.g. in the metric space $(l_{\infty},\|\cdot\|_{\infty})$, some intervals are not convex.

\begin{thm}\label{chybridB'c}
	Let $x,y$ be two distinct points in a complete, Menger convex metric space $(X,d)$ with the property $(B')$, and the distance function $d$ be Menger convex. If the set $[x,y]$ is convex, then an $(\alpha, \beta, \gamma)-$generalized hybrid $T:[x,y]\to [x,y]$ has a fixed point.
\end{thm}
\begin{proof}
	The proof follows from Theorems \ref{hybridB'} and \ref{[x,y]}.
\end{proof}

Given a pair of distinct points $x,y$ in a Takahashi convex metric space $X$,  we write $\langle x, y \rangle_W = \{W(x,y,\lambda):\lambda \in (0,1)\}$, and
$[x,y]_W=\langle x, y \rangle_W \cup \{x,y\}$. A normed space $X$ is a Takahashi convex metric space, and for $x,y\in X,x\neq y,$ $t_1,t_2\in [0,1], t_1<t_2$; the points $z_1=t_1x+(1-t_1)y$ and $z_2=t_2x+(1-t_2)y$ lie in $[ x,y]_W$, and, $z_1\in  [z_2,y] _W $, $z_2\in  [z_1,x] _W $.
Also a CAT(0) sapce $X$ (see \cite{sm09}, \cite{pc06}, \cite{sd08}) is a Takahashi convex metric space with a Takahashi convex structure $tx\oplus(1-t)y$, and in this space the same betweenness property of the between points of $x,y$ holds, i.e., if $x,y\in X,x\neq y,$ $t_1,t_2\in [0,1], t_1<t_2$, then the points $z_1=t_1x\oplus(1-t_1)y$ and $z_2=t_2x\oplus(1-t_2)y$ lie in $ [x,y]_W$, and, $z_1\in  [z_2,y] _W $, $z_2\in  [z_1,x] _W $. Although, in a Menger convex metric space, this betweenness property generally does not hold. For example,
\begin{example} \label{ex1} 
	Let $x=(0,0,0,...), y=(2,0,0,0,...)$ be in the Menger convex space $X=(l_{\infty},\|\cdot\|_{\infty})$, and $t_1=\frac{1}{4}$, $t_2=\frac{1}{2}$. Then $S^{\cap}_{t_1}:= S[x,t_1d(x,y)]\cap S[y,(1-t_1)d(x,y)]
	\supset \{(x_i)\in X: x_1=\frac{1}{2},\ \frac{-1}{2}\leq x_i\leq \frac{1}{2}\ \forall\  i\geq 2\}$, and 
	$S^{\cap}_{t_2}:=S[x,t_2d(x,y)]\cap S[y,(1-t_2)d(x,y)]\supset \{(x_i)\in X: x_1=1,\ -1\leq x_i\leq 1\ \forall\  i\geq 2\}.$
	Let $x^1=(\frac{1}{2},0,0,0,...)\in S^{\cap}_{t_1}\ \text{and}\ x^2=(1,1,1,...)\in S^{\cap}_{t_2}$. Then,
	$d(x^1,x^2)+d(x^2,y)=2\neq d(x^1,y)=\frac{3}{2}$, and $d(x,x^2)+d(x^2,x^1)=2\neq d(x,x^1)=1/2$. This implies,
	$x^2\not\in [x,x^1]\cup[x^1,y]$. Also we have, $x^1\not \in [x,x^2]\cup [x^2,y]$.
	
\end{example}
Thus we have the following: 
\begin{prop}
	Let $x,y$ be a pair of distinct points in a Menger convex space, and $t_1,t_2\in (0,1)$. Consider $x^1\in S[x,t_1d(x,y)]\cap S[y,(1-t_1)d(x,y)]$ and $x^2\in S[x,t_2d(x,y)]\cap S[y,(1-t_2)d(x,y)]$. Then, $x^1$ may not be in the metric interval $[x,x^2]$ or $[x^2,y]$.
\end{prop}

\section{Conclusion}
We explored some fundamental results for Menger convex structure. We showed that the uniformly Takahashi convex metric spaces satisfy the property $(A)$, and therefore, they are uniformly Menger convex as well. As an immediate deduction, we also found that the uniformly convex normed spaces have the property $(A)$. We proved that the metric intervals in a complete Menger convex space are complete. For a metric $d$ on a real vector space $X$, satisfying $d(u,\lambda x + (1-\lambda)y)\leq \lambda d(u,x) + (1-\lambda)d(u,y)$ for all $x,y,u\in X, \lambda \in [0,1],$ we showed that $d$ is indeed induced by a norm.

Using the Hausdorff metric, we took a larger domain for the definition of the property $(B)$, and called this extended version of the definition as the property $(B')$. We achieved a nonempty intersection property for decreasing sequences of nonempty, bounded, closed, convex subsets in a complete Menger convex space with property $(B')$, which also implies that the complete Menger convex spaces with property $(B')$ are reflexive. Using the obtained nonempty intersection property, we acquired a fixed point result for a mapping weaker than $(\alpha, \beta)$-generalized hybrid mappings, defined on a Menger convex space with property $(B')$.

\end{document}